\documentclass[12pt]{amsart}
\usepackage{amsmath, amssymb, amstext, amsthm, eqnarray, amscd}
\usepackage{tikz}
\usepackage{tikz-cd}
\usetikzlibrary{%
  matrix,%
  calc,%
  arrows%
}
\usepackage{fullpage}

\usepackage{float}
\usepackage{mathtools}
\usepackage{mathabx}
\usepackage{graphicx}
\usepackage[mathscr]{euscript}
\usepackage{enumerate}
\usepackage{hyperref}
\usepackage{pgfkeys}
\usepackage{fullpage}
\usepackage{pst-node}
\usepackage{lastpage}
\usepackage{fancyhdr}
\usepackage{multirow}
\allowdisplaybreaks
\usepackage{graphicx}
\usepackage{tikz,color,soul}
\usepackage{chngcntr}
\usepackage{stmaryrd}
\usepackage{tabularx}

\usepackage{enumitem}
\usepackage{kantlipsum}
\usepackage{array}
\usepackage{cleveref}
\usepackage[english]{babel}
\usepackage{enumitem}
\def\multiset#1#2{\ensuremath{\left(\kern-.3em\left(\genfrac{}{}{0pt}{}{#1}{#2}\right)\kern-.3em\right)}}

\def\lcm{\mathrm{lcm}}

\def\F{\mathcal F}

\def\h{\mathrm{ht}}

\def\reesI{\mathcal{R}(\mathcal{I})}

\newtheorem{theorem}{Theorem}[section]

\newtheorem{cor}[theorem]{Corollary}
\theoremstyle{definition}
\newtheorem{definition}[theorem]{Definition}
\newtheorem{eg}[theorem]{Example}
\newtheorem{remark}[theorem]{Remark}

\numberwithin{equation}{section}

\def\lcm{\mathrm{lcm}}

\def\F{\mathcal F}

\def\h{\mathrm{ht}}

\def\defect{\mathrm{def}}

\def\filJ{\mathcal{J}}

\newcommand{\nothing}[1][]{%
    \tikz[overlay,remember picture]{
    \draw[red]
      ($(l4)+(-0.7em,0.7em)$) rectangle
      ($(r4)+(0.5em,-0.5em)$);}
}

\title{Defect Functions Between Filtrations of Ideals}

\author{Arindam Banerjee}
    \address{Department of Mathematics, Indian Institute of Technology Kharagpur, Kharagpur, INDIA - 721302.}
    \email{123.arindam@gmail.com}

\author{T\`ai Huy H\`a }
    \address{Department of Mathematics,Tulane University, 6823 St. Charles Ave., New Orleans,
LA 70118}
    \email{tha@tulane.edu}

\author{Vivek Bhabani Lama}
    \address{Department of Mathematics, Indian Institute of Technology Kharagpur, Kharagpur, INDIA - 721302.}
    \email{ v.bhabani.lama@gmail.com, vivekbhabanilama@kgpian.iitkgp.ac.in}


\date{}

\makeindex
\begin{document}
\begin{abstract}
We introduce and study the \emph{defect function} associated to a pair of filtrations of ideals, which generalizes the symbolic defect of ideals. Under the assumption that the Rees algebra of one filtration is Noetherian and that a natural graded module measuring the interaction between the filtrations is finitely generated over it, we show that the corresponding defect function is asymptotically a quasi-polynomial. Moreover, the defect function becomes eventually polynomial when the Rees algebra of the first filtration is standard graded.
For filtrations arising from saturations and ordinary powers of monomial ideals, we further analyze the structure of the quasi-polynomial. We prove that the top two coefficients of the eventual quasi-polynomial are constant under natural hypotheses.
\end{abstract}
\maketitle


\section{Introduction}\label{section intro}

The comparison between symbolic and ordinary powers of ideals has long been a central theme in commutative algebra and algebraic geometry. Symbolic powers encode geometric vanishing conditions, while ordinary powers reflect algebraic multiplicities, and understanding the nature of their discrepancy forms a key part of the study of blowups, Rees algebras, and asymptotic invariants (cf. \cite{BH,DHST, ELS,GHVT,HH}). A particularly useful measure of this discrepancy is the \emph{symbolic defect}, introduced in \cite{galetto2019symbolic}, which records the minimal number of generators of the module $I^{(m)}/I^m$ for a given ideal $I$. This invariant has been studied in various contexts, including monomial ideals, ideals of fat point schemes, and ideals defining singular varieties, and it plays an important role in the investigation of containment problems and asymptotic phenomena (cf.~\cite{BD,GHM,BO}, among others).

Symbolic defect fits naturally into a broader picture. Many constructions in commutative algebra  such as symbolic powers, ordinary powers, integral closures, saturations, colon filtrations, tight closure sequences, and mixed power sequences --- give rise to filtrations of ideals. It is natural to compare two such filtrations and measure their discrepancy. Motivated by this viewpoint, we introduce the \emph{defect function} for pairs of filtrations of ideals. Given filtrations $\mathcal{I} = \{I_n\}_{n \ge 0}$ and $\filJ = \{J_n\}_{n \ge 0}$ of ideals in a ring $R$, which is either a Noetherian local ring or a standard graded algebra over a field, we define
\[
\defect(\mathcal{I},\filJ,n)
= \mu\!\left( \frac{I_n + J_n}{J_n} \right),
\]
which recovers symbolic defect when $\mathcal{I}$ is the symbolic power filtration and $\filJ$ is the ordinary power filtration (see \Cref{definition of defect}). This framework captures a wide range of natural examples and allows a unified treatment of several defect-type invariants.

A second major theme in commutative algebra is the study of \emph{asymptotic behavior} of algebraic invariants. Prior work, including \cite{cutkosky1999asymptotic, DHHT, Kodiyalam, NguyenTrung, TW} and references therein, show that invariants such as Castelnuovo--Mumford regularity, Hilbert functions, and depth often stabilize or become eventually linear or quasi-polynomial. Recent work has extended these asymptotic phenomena to invariants such as the $v$-number \cite{conca2024note,FG2025}, Waldschmidt constants, and asymptotic resurgence \cite{BH,DFMS, HKNN}. In the context of symbolic defect, Drabkin and Guerrieri \cite{BD} proved that $\mathrm{sdef}(I,n)$ is eventually a quasi-polynomial provided the symbolic Rees algebra of $I$ is Noetherian. Their proof relies on Hilbert--Serre theory applied to a finitely generated graded module over the symbolic Rees algebra.

Our first main result generalizes the theorem of Drabkin--Guerrieri to the setting of arbitrary filtrations. 
\medskip
\\ 
\noindent\textbf{Theorem~\ref{main theorem 1}.}
\emph{Let $\mathcal{I}=\{I_i\}_{i \geq 0}$ and ${\mathcal{J}}=\{J_{i}\}_{i \geq 0}$ be (homogeneous) filtrations of ideals in $R$. Assume that the Rees algebra $\mathcal{R}(\mathcal{I})$ is a Noetherian and the defect module 
$$N=\displaystyle \bigoplus_{n \geq 0} \dfrac{I_n+J_n}{J_n}$$  
is a finitely generated module over $\mathcal{R}(\mathcal{I})$. Then, $\defect(\mathcal{I},\filJ,n)$ is eventually a quasi-polynomial in $n$ with period $\alpha=\mathrm{lcm}(\alpha_{1},\alpha_{2},\dots,\alpha_{s})$, where $\alpha_1,\alpha_2,\dots,\alpha_s$ are the degrees of the generators of $\reesI$. Finally, if $\mathcal{R}(\mathcal{I})$ is standard graded as an algebra over $R$, then $\defect({\mathcal{I},{\mathcal{J}} ,n)}$ is eventually a polynomial in $n$.}

\medskip

Theorem \ref{main theorem 1} can be generalized by replacing $N_n$ with $\mathbf{F}(N_n)$, where $\mathbf{F}$ is a functor on the category of finitely generated $R$-modules which commutes with direct sums, and considering $\mu(\mathbf{F}(N_n))$ instead of $\mu(N_n)$; see \Cref{Theorem on Functors}. Particularly, \Cref{cor.TorExt} shows that for any finitely generated $R$-module $M$ and any fixed $i \ge 0$, $\mu(\text{Tor}^R_i(M,N_n))$ is asymptotically quasi-polynomial in $n$. 


In the case of monomial ideals, further structure emerges from the geometry of exponent sets and Newton polyhedra (cf.~\cite{BG2009,D+, DO2025, HerzogHibi, BO, HoaTrung}). For filtrations of monomial ideals, in Theorem \ref{growth of sdef vs idef}, we provide sufficient conditions ensuring that the growth of the defect function for $(\mathcal{I},\filJ)$ is comparable to that of the defect function for their integral closures. This generalizes \cite[Proposition~4.5]{BO} and extends the philosophy that monomial filtrations inherit linear or quasi-linear asymptotic behavior from the combinatorics of their exponent sets.

Our second main theorem concerns the defect between the saturation filtration of a monomial ideal $I$ with respect to another monomial ideal $J$ and the filtration of ordinary powers of $I$. For this particular pair of filtrations, we determine the top two coefficients of the eventual quasi-polynomial describing the defect.

\medskip
\noindent\textbf{Theorem~\ref{theorem on coefficients}.}
\emph{Let $R=\mathbb{K}[x_1,\dots,x_r]$ be a standard graded polynomial ring. Let $I$ and $J$ be  monomial ideals. Consider two filtrations $\mathcal{I}=\{I^n: J^\infty\}_{n \geq 0}$ and $\filJ=\{I^n\}_{n \geq0}$ of ideals in $R$. Assume that the Rees algebra $\mathcal{R}(\mathcal{I})$ is Noetherian and the defect module 
$$N=\displaystyle \bigoplus_{n \geq 0} \dfrac{I_n+J_n}{J_n}$$  
is a finitely generated module over $\mathcal{R}(\mathcal{I})$. Suppose further that the quasi-polynomial (by \Cref{main theorem 1}) defect function of $(\mathcal{I},\filJ)$ is of the form 
$$\defect(\mathcal{I},\filJ,n) = a_c(n)n^c+a_{c-1}(n)n^{c-1}+ \text{ lower order terms},$$ 
for $n\gg 0$, where $c$ is the degree. Then, the following statements hold:
  \begin{enumerate}
      \item The coefficient $a_c(n)$ is constant.
      \item If $I$ is generated by monomials in the same degree and $\h(I) \geq 2$, then the non-zero values of $a_{c-1}(n)$ are constant.   
  \end{enumerate} }

\medskip

\noindent
\textbf{Organization of the paper.}
Section~\ref{section prelims} collects preliminary material on filtrations, Rees algebras, Hilbert functions and symbolic defect. In Section~\ref{section 3}, we prove Theorem~\ref{main theorem 1}, establish quasi-polynomial behavior of defect functions. Section~4 turns to monomial filtrations: we compare defect functions with their integral-closure analogues and prove Theorem~\ref{theorem on coefficients} concerning the leading coefficients of saturation defect for monomial ideals.


\section{Preliminaries}\label{section prelims}
In this section, we introduce various notations and auxiliary results that will be used. Throught the paper, let $\mathbb{K}$  be a field and let $R$ be either a Noetherian local ring with maximal ideal $\mathfrak{m}$ or a standard graded $\mathbb{K}$-algebra with maximal homogeneous ideal $\mathfrak{m}$. 

\subsection{Filtration of Ideals} We shall recall the definition of a filtration of ideals, and discuss various rings and algebras that are associated to filtrations. 

\begin{definition}
 A collection $\mathcal{I}=\{I_i\}_{i \in \mathbb{N}}$ of ideals in $R$ is called a \emph{filtration} if the following conditions hold:
 \begin{enumerate}
     \item it is a \emph{graded family} of ideals, i.e., $I_0 = R$ and, for every $m,n \in \mathbb{N}$, $I_mI_n \subseteq I_{m+n}$; and 
     \item for any $n \in \mathbb{N}$, we have $I_{n+1} \subseteq I_n.$
 \end{enumerate}
\end{definition}

Given a filtration $\mathcal{I}=\{I_i\}_{i \in \mathbb{N}}$ of ideals in $R$, its \emph{Rees algebra} is defined by 
$$\reesI=\displaystyle\bigoplus _{n=0}^{\infty}I_nt^n \subseteq R[t],$$ 
and its \emph{special fiber ring} is given by 
$$\mathcal{F}(\mathcal{I})= \reesI \otimes_R \frac{R}{\mathfrak{m}}.$$

A filtration $\mathcal{I} = \{I_i\}_{i \in \mathbb{N}}$ in a graded ring $R$ is called \emph{homogeneous} if every $I_i$ is a homogeneous ideal. When $R$ is a polynomial ring, we call $\mathcal{I}$ a \emph{monomial filtration} if each $I_i$ is a monomial ideal.




\subsection{Hilbert Functions and Asymptotic Behaviour} We now define a quasi-polynomial and recall the Hilbert-Serre theorem in the theory of Hilbert function.

\begin{definition}
    A \emph{quasi-polynomial} function in $\mathbb{N}$ is a function $f: \mathbb{N}\rightarrow \mathbb{Q}$ such that for some $p \in \mathbb{N}$, there exist polynomials $f_0,f_1,\dots,f_{n-1}$ with rational coefficients having the property that for all $m \in \mathbb{N}$, $f(m)=f_r(m)$ whenever $m \equiv r \ (\text{mod } p).$ The value of $p$ is called the \emph{period} of $f$.
\end{definition}

\begin{theorem}[{\cite[4.4.1]{stanley2012enumerative}}]\label[theorem]{quasi-poly}
 If a numerical function $\phi: \mathbb{N}\rightarrow \mathbb{Z}$ has the generating function
 $$\sum_{n=0}^{\infty}\phi(n)t^n=\dfrac{q(t)}{\displaystyle \prod_{i=1}^s(1-t^{d_i})}$$
 for some $d_1,d_2,\dots,d_s \in \mathbb{N}$, and some polynomial $q(t)\in \mathbb{Z}[t]$, then $\phi(n)=Q(n)$ for sufficiently large $n$, where $Q(n)$ is a quasi-polynomial with period $p=\lcm(d_1,d_2,\dots,d_s)$. 
\end{theorem}

The following theorem is a well-known result in the theory of Hilbert-function of graded modules over a graded ring $R$. 

\begin{theorem}[Hilbert-Serre] \label[theorem]{HF is quasi polynomial}
    Let $M=\bigoplus_{n \in \mathbb{Z}}M_i$ be a finitely generated graded module of dimension $d$ over a graded ring $R=\bigoplus_{i \in \mathbb{Z}}R_i$. Then, there exists a quasi-polynomial $Q_M$ of degree $d-1$ such that the Hilbert function of $M$ satisfies $H_M(n)=Q_M(n)$ for $n\gg 0$. Moreover, the period of $Q_M(n)$ divides $\lcm(\deg r_1,\dots, \deg r_d)$  where $r_1,\dots,r_d $ is a homogeneous system of parameters of $M$.  
\end{theorem}

When the ring $R$ is a \emph{standard} graded ring, the following theorem shows that the Hilbert-function is a polynomial for large values of $n$.  

\begin{theorem}[{\cite[Theorem~13.2.]{matsumura1989commutative}}] \label{poly}
 Let $R$ be a standard graded algebra over a field $\mathbb{K}$, and let $M$ be a finitely generated graded $R$-module of dimension $d$. There exists a polynomial $\phi_M(X)$ of degree $d-1$ such that, for $n\gg 0$, we have $\dim_{\mathbb{K}}(M_n)=\phi_M(n)$.    
\end{theorem}
  



\section{Defect functions of pairs and asymptotic behaviour}\label{section 3}
In this section, we introduce the notion of the defect function for a pair of filtrations $(\mathcal{I},\filJ)$ of ideals in $R$ (see Definition~\ref{definition of defect}). This invariant measures the ``difference in containment’’ between the two filtrations. We show that when the Rees algebra of $\mathcal{I}$ is Noetherian and the associated \emph{defect module} is finitely generated over this Rees algebra, the defect function is eventually a quasi-polynomial (and a polynomial in the standard graded case; see \Cref{main theorem 1}). 

Recall that throughout $R$ will denote either a Noetherian local ring or a standard graded algebra over a field.

\begin{definition}\label[definition]{definition of defect}
Let $\mathcal{I}=\{ I_i\}_{i \geq 0}$ and ${\mathcal{J}}=\{ J_{i}\}_{i \geq 0}$ be (homogeneous) filtrations of ideals in $R$.  The \textit{defect function} of the pair $(\mathcal{I},\mathcal{J})$ is defined by 
$$\defect(\mathcal{I},{\mathcal{J}},n)=\mu\left( \frac{I_{n}+J_{n}}{J_{n}} \right).$$ 
The module $N=\displaystyle\bigoplus_{n \geq 0}\dfrac{I_n+J_n}{J_n}$ is called the \emph{defect module} of $(\mathcal{I},\filJ)$. 
\end{definition}

The following example illustrates the notion of symbolic function in Definition \ref{definition of defect}.

\begin{eg}

Let $R=\mathbb{K}[x_1,x_2,x_3]$, and let $P = (x_2,x_3)$ and
$I = (x_1x_2, x_2x_3, x_3x_1)$ be ideals in $R$.
Consider the following filtrations of ideals:
$$\mathcal{I} = \left\{I^n : P^\infty\right\}_{n \in \mathbb{N}} \text{ and } \filJ = \left\{I^{(n)}\right\}_{n \in \mathbb{N}}.$$
Clearly, $I^n:P^\infty=(x_1,x_2)^n \cap (x_1,x_3)^n$. We claim that 
$$(x_1,x_2)^n \cap (x_1,x_3)^n=(x_1^{n-i}x_2^ix_3^i \ |\  i=0,1,2,\dots,n).$$
Indeed, the containment $(x_1^{n-i}x_2^ix_3^i \ |\  i=0,1,2,\dots,n) \subseteq(x_1,x_2)^n \cap (x_1,x_3)^n$ is a straightforward computation. On the other hand, suppose $m \in (x_1,x_2)^n \cap (x_1,x_3)^n$ is a monomial in $R$. Then, $m$ is divisible by $x_1^{n-i_1}x_2^{i_1}$ and $x_1^{n-i_2}x_3^{i_2}$ for some $i_1,i_2 \leq n$. Set $j=\mathrm{min}\{i_1,i_2\}$ then   $x_1^{n-j}x_2^{j}$ and $x_1^{n-j}x_3^{j}$ divide $m$. Therefore, $m \in(x_1^{n-i}x_2^ix_3^i \ |\  i=0,1,2,\dots,n)$.\\ For $n=1$, the only generator of $\frac{I:P^\infty}{I}$ is given by $(x_1)$ and hence $\mu\left(\dfrac{I:P^\infty}{I}\right )=1.$ Now, by \cite[Lemma~4.1]{BCMM}, for $n \ge 2$, we get
$$I^{(n)}=I^n+(D(n)),$$ where $D(n)=\{x_1^{a_1}x_2^{a_2}x_3^{a_3} \ | \ a_{i_1}+a_{i_2} \geq n \text{ for } \{i_1,i_2\} \subseteq \{1,2,3\} \text{ and } a_1+a_2+a_3 \leq 2n-1\}.$ This implies that 
$$\defect(\mathcal{I},\filJ,n) =  \mu\left(\dfrac{I^n:P^\infty}{I^{(n)}}\right)= \left\lceil\dfrac{n}{2}\right\rceil.$$
\end{eg}


\begin{remark}
 Let $I, J \subseteq R$ be ideals.
\begin{enumerate}
    \item Consider $\mathcal{I} = \{I^{(n)}\}_{n \geq 0}$ and $\filJ =\{I^n\}_{n \geq 0}$. Then, the defect function of $(\mathcal{I},\filJ)$ is given by 
    $$\defect(\mathcal{I},\filJ,n)= \mu\left(\frac{I^{(n)}}{I^n}\right),$$
    which is a well-studied under the nomenclature of \emph{symbolic defect} of $I$ and is often denoted by $\text{sdef}(I,n)$.  
    \item Consider $\mathcal{I} = \{\overline{I^{n}}\}_{n \geq 0}$ and $\filJ  =\{I^n\}_{n \geq 0}$. Then, the defect function of $(\mathcal{I},\filJ)$  is 
     $$\defect(\mathcal{I},\filJ,n)= \mu\left(\frac{\overline{I^{n}}}{I^n}\right),$$
     which is closely related to the symbolic defect of $I$, but has not been much investigated.
     \item Let $\mathcal{I}= (I^{(n)}:J^\infty)_{n \geq 0 }$ and $\filJ = \{I^n\}_{n \geq 0}$. Then, the defect function of $(\mathcal{I},\filJ)$ is $$\defect(\mathcal{I},\filJ,n) =\mu\left(\dfrac{(I^{n}:J^\infty)}{I^n}\right).$$
\end{enumerate} 
\end{remark}

We now prove our first main result regarding the asymptotic behaviour of the defect function. We show that under certain conditions on the Rees algebra of $\mathcal{I}$ and the defect module between $\mathcal{I}$ and $\filJ$, the defect function is eventually a quasi-polynomial (respectively, a polynomial in the standard graded case). 

\begin{theorem}\label[theorem]{main theorem 1}
Let $\mathcal{I}=\{I_i\}_{i \geq 0}$ and ${\mathcal{J}}=\{J_{i}\}_{i \geq 0}$ be (homogeneous) filtrations of ideals in $R$. Assume that the Rees algebra $\mathcal{R}(\mathcal{I})$ is a Noetherian and the defect module 
$$N=\displaystyle \bigoplus_{n \geq 0} \dfrac{I_n+J_n}{J_n}$$  
is a finitely generated module over $\mathcal{R}(\mathcal{I})$. Then, $\defect(\mathcal{I},\filJ,n)$ is eventually a quasi-polynomial in $n$ with period $\alpha=\mathrm{lcm}(\alpha_{1},\alpha_{2},\dots,\alpha_{s})$, where $\alpha_1,\alpha_2,\dots,\alpha_s$ are the degrees of the generators of $\reesI$. Finally, if $\mathcal{R}(\mathcal{I})$ is standard graded as an algebra over $R$, then $\defect({\mathcal{I},{\mathcal{J}} ,n)}$ is eventually a polynomial in $n$.
\end{theorem}

\begin{proof}
 We claim that, for any $n \in \mathbb{N}$, $\mu\left(\dfrac{I_{n}+J_{n}}{J_{n}}\right)=\dim_{\mathbb{K}}\left(\dfrac{I_{n}+J_{n}}{\mathfrak{m}I_{n}+J_{n}}\right)$. Indeed, consider the following isomorphism of $R$-modules:
 $$\frac{R}{\mathfrak{m}} \otimes_R \left(\frac{I_{n}+J_{n}}{J_{n}}\right) \cong \dfrac{\left( \frac{I_{n}+J_{n}}{J_{n}}\right)}{\mathfrak{m} \left(\frac{I_{n}+J_{n}}{J_{n}}\right)} \cong \dfrac{I_{n}+J_{n}}{\mathfrak{m}(I_{n}+J_{n})+J_{n}}=\dfrac{I_{n}+J_{n}}{\mathfrak{m}I_{n}+J_{n}}.$$
 Hence, by Nakayama's lemma, we get the desired equality 
 $$\mu\left(\dfrac{I_{n}+J_{n}}{J_{n}}\right)=\dim_{\mathbb{K}}\left(\dfrac{I_{n}+J_{n}}{\mathfrak{m}I_n+J_n}\right).$$ 
 
 Set $M=\displaystyle\bigoplus _{n \in \mathbb{N}}M_n=\displaystyle\bigoplus_{n\in \mathbb{N}}\dfrac{I_n+J_n}{{\mathfrak{m}}I_n+J_n}$. Then, $M$ is a finitely generated graded module over the special fiber cone $\mathcal{F}(\mathcal{I})=\dfrac{\mathcal{R}(\mathcal{I})}{{\mathfrak{m}}\mathcal{R}(\mathcal{I})}$, in which $(\mathcal{F}(\mathcal{I}))_0=\mathbb{K}$; the module structure of $M$ over $ \F(\mathcal{I})$ is inherited from that of $N$ over $\reesI$.  
 Thus,  $h_M(z)=\displaystyle\sum_{n=0}^{\infty}\defect(\mathcal{I},\filJ,n)z^n$ is eventually a quasi-polynomial, by \Cref{HF is quasi polynomial}. 
 
 Furthermore, if the Rees algebra $\mathcal{R}(\mathcal{I})$ is standard graded, then so is $\mathcal{F}(\mathcal{I}) = \mathcal{R}(I)/\mathfrak{m}\mathcal{R}(I)$. Therefore, $h_M(z)=\displaystyle\sum_{n=0}^{\infty}\defect(\mathcal{I},\filJ,n)z^n$ is asymptotically a polynomial in $n$ with period $\alpha=\mathrm{lcm}(\alpha_{1},\alpha_{2},\dots,\alpha_{s})$, by \Cref{poly}. 
\end{proof}

The following example illustrates \Cref{main theorem 1}.

\begin{eg}[{\cite[Example~4.12]{MM}}] \label[eg]{Example1}
Let $I=(x_1x_2,x_2x_3,x_3x_1,x_1x_4) \subseteq \mathbb{K}[x_1, \dots, x_4]$. Then, for the filtrations $\mathcal{I}=\{I^{(n)}\}_{n \geq 0}$ and $\filJ=\{I^n\}_{n \geq 0}$, their defect function is     $$\defect(\mathcal{I},\filJ,n)=\begin{cases}
       2k^2-1, & n=2k\\
       2k^2+2k, & n =2k+1\\
   \end{cases}$$  
\end{eg}

The following corollaries are direct consequences of \Cref{main theorem 1}.

\begin{cor}
 Let $a,b$ be positive integers. For the filtrations $\mathcal{I}=\{I^{an}\}_{n \geq 0}$ and $\filJ =\{I^{(bn)}\}_{n \geq 0}$ where $I$ is a (homogeneous) ideal in $R$, the defect function 
 $$\defect(\mathcal{I},\filJ,n)= \mu\left(\frac{I^{an}+ I^{(bn)}}{I^{(bn)}}\right)$$ 
 is eventually a polynomial in $n$. 
\end{cor}
\begin{proof}
    The assertion follows from \Cref{main theorem 1}, noting that the Rees algebra $\mathcal{R}(\mathcal{I})$ is a standard graded algebra over $R$.  
\end{proof}

\begin{cor}   
   Let $I$, $J$ be (homogeneous) ideals in $R$. For the filtrations $\mathcal{I}=\{(I^n:J^\infty)\}_{n \geq 0}$ and $\filJ =\{I^n\}_{n \geq 0}$, if the Rees algebra $\mathcal{R}(\mathcal{I})$ is Noetherian, then the defect function $$\defect(\mathcal{I},\filJ,n)= \mu\left(\frac{(I^n:J^\infty)}{I^n}\right)$$ 
   is eventually a quasi-polynomial in $n$.   
\end{cor}
\begin{proof}
    The assertion follows from \Cref{main theorem 1}, under the assumption that the Rees algebra $\mathcal{R}(\mathcal{I})$ is Noetherian.  
\end{proof}

\Cref{main theorem 1} also recovers a known result on the eventual quasi-polynomial behaviour of the symbolic defect of an ideal in a local ring or a homogeneous ideal in a standard graded algebra over a field.
\begin{cor}[{\cite[Theorem~2.4]{BD}}]
Let $I$ be a (homogeneous) ideal in $R$ such that the symbolic Rees algebra $\mathcal{R}_s(I) = \bigoplus_{n \ge 0}I^{(n)}t^n$ of $I$ is a Noetherian ring, and let $d_1, \ldots, d_s$ be the degrees of the minimal generators of $\mathcal{R}_s(I)$ as an $R$-algebra. Then, the symbolic defect
$$\operatorname{sdef}(I,n) = \mu\left(\frac{I^{(n)}}{I^n}\right)$$ 
is eventually a quasi-polynomial in $n$ with quasi-period $d = \operatorname{lcm}(d_1, \ldots, d_s)$.
\end{cor}
\begin{proof}
    The assertion follows from \Cref{main theorem 1} for filtrations $\mathcal{I}=\{I^{(n)}\}_{n \geq 0}$ and $\filJ =\{I^n\}_{n \geq 0}$, where the Rees algebra $\mathcal{R}(\mathcal{I})$ is Noetherian by the hypotheses.  
\end{proof}

\Cref{main theorem 1} holds in greater generality, where the graded piece $N_n$ of the defect module $N=\displaystyle\bigoplus_{n \geq 0}\frac{I_n+J_n}{J_n}$ is replaced by $\mathbf{F}(N_n)$ for a functor $\mathbf{F}$. 

\begin{theorem}\label[theorem]{Theorem on Functors}
  Let $\mathcal{I}=\{I_i\}_{i \geq 0}$ and ${\mathcal{J}}=\{J_{i}\}_{i \geq 0}$ be filtrations of ideals in $R$. Assume that the Rees algebra $\mathcal{R}(\mathcal{I})$ is a Noetherian and the defect module $N=\displaystyle \bigoplus_{n \geq 0} \dfrac{I_n+J_n}{J_n}$  is a finitely generated module over $\mathcal{R}(\mathcal{I})$. Let  $\mathbf{F}: R\text{-}\mathrm{Mod}^{\mathrm{fg}} \longrightarrow R\text{-}\mathrm{Mod}^{\mathrm{fg}}$ be a functor on the category of finitely generated $R$-modules which commutes with direct sums, i.e., $\displaystyle\mathbf{F}\left(\bigoplus_{\lambda \in \Lambda}M_\lambda\right)=\bigoplus_{\lambda \in \Lambda}\mathbf{F}(M_\lambda)$. Then, $\mu\left(\mathbf{F}(N_n)\right)$ is eventually a quasi-polynomial in $n$ with period $\alpha=\mathrm{lcm}(\alpha_{1},\alpha_{2},\dots,\alpha_{s})$, where $\alpha_1,\alpha_2,\dots,\alpha_s$ are the degrees of the generators of $\reesI$. Furthermore, if $\mathcal{R}(\mathcal{I})$ is standard graded as an algebra over $R$, then $\mu\left(\mathbf{F}(N_n)\right)$ is eventually a polynomial in $n$.
\end{theorem} 
\begin{proof} The proof goes in the same line of arguments as that of \Cref{main theorem 1}. Particularly, let $N_n=\dfrac{I_n+J_n}{I_n}$. Similar to what was done in \Cref{main theorem 1}, by considering 
$$\mathbf{F}(N_n)\otimes_R \dfrac{R}{\mathfrak{m}} \cong \dfrac{\mathbf{F}(N_n)}{\mathfrak{m}\mathbf{F}(N_n)},$$ 
and applying Nakayama's lemma, we get
$$\mu\left(\mathbf{F}\left(N_n\right)\right)=\dim_{\mathbb{K}}\left(\dfrac{\mathbf{F}(N_n)}{\mathfrak{m}\mathbf{F}(N_n)}\right).$$ 
  Since $N$ is a finitely generated module over $\mathcal{R}(\mathcal{I})$,  $\displaystyle\bigoplus_{n \geq 0}\left(\dfrac{\mathbf{F}(N_n)}{\mathfrak{m}\mathbf{F}(N_n)}\right) $ is a finitely generated graded module over the special fiber cone $\mathcal{F}(\mathcal{I})$. The conclusion now follows from Theorems \ref{HF is quasi polynomial} and \ref{poly}.
\end{proof}

\Cref{Theorem on Functors} gives the following corollaries.

\begin{cor}\label[cor]{cor.TorExt}
  Let $\mathcal{I}=\{I_i\}_{i \geq 0}$ and ${\mathcal{J}}=\{J_{i}\}_{i \geq 0}$ be filtrations of ideals in $R$. Suppose that the Rees algebra $\mathcal{R}(\mathcal{I})$ is Noetherian and the defect module $N=\displaystyle \bigoplus_{n \geq 0} \dfrac{I_n+J_n}{J_n}$  is a finitely generated module over $\mathcal{R}(\mathcal{I})$. Then, the following statements hold:
    \begin{enumerate}
        \item For any multiplicatively closed set $S$ of the ring $R$, $\mu\left(S^{-1}(N_n)\right)$ is eventually a quasi-polynomial in $n$. 
        \item For $R$-modules $M_1, M_2, \dots, M_k$ and $i_1, i_2, \dots, i_k \in \mathbb{N}$,
\[
\mu\!\left(
\mathrm{Tor}^R_{i_1}\Big(M_1,\ 
\mathrm{Tor}^R_{i_2}\big(M_2,\ 
\dots\, 
\mathrm{Tor}^R_{i_k}(M_k, N_n)
\dots
\big)
\Big)
\right)
\]
is eventually a quasi-polynomial in $n$.
    \end{enumerate}
    Furthermore, in these situations, the period of the quasi-polynomial is $\alpha=\mathrm{lcm}(\alpha_{1},\alpha_{2},\dots,\alpha_{s})$, where the numbers $\alpha_1,\alpha_2,\dots,\alpha_s$ are the degrees of the generators of $\reesI$. In addition, if $\mathcal{R}(\mathcal{I})$ is standard graded as an algebra over $R$ then these quasi-polynomials are in fact polynomials.
\end{cor}
\begin{proof}
(1) The assertion follows directly from \Cref{Theorem on Functors}, noting that
$$S^{-1}\left(\bigoplus_{n \geq 0}\dfrac{I_n+J_n}{J_n}\right)=\bigoplus_{n \geq 0}S^{-1}\left(\dfrac{I_n+J_n}{J_n} \right).$$

(2) The conclusion is also a direct consequence of \Cref{Theorem on Functors}, due to the fact that
    $$\mathrm{Tor}_i^R\left(M,\displaystyle\bigoplus_{n \geq 0} \dfrac{I_n+J_n}{J_n}\right)\cong\displaystyle\bigoplus_{n \geq 0}\mathrm{Tor}_i^R\left(M, \dfrac{I_n+J_n}{J_n}\right)$$
for any $R$-module $M$ and any $i \in \mathbb{N}$.
\end{proof}

\begin{remark}
    We do not know if a similar statement to \Cref{cor.TorExt} holds for the Ext modules. The existence of such a statement depends on the commutativity between Ext and direct sums, which has been much investigated (see, for instance \cite{Breaz2013} and \cite{WhenEXT}).
\end{remark}

\section{Monomial ideals and leading coefficients of defect functions}
In this section we study the defect function for pairs of \textit{monomial filtrations} (i.e., filtrations in which every ideal is a monomial ideal) in the polynomial ring $R$. Our first goal is to compare the growth of the defect function with that of the defect arising from the integral-closure filtrations; see \Cref{growth of sdef vs idef}. This extends the corresponding result for symbolic defect obtained in \cite[Proposition~4.5]{BO}. We then establish our second main theorem, showing that for the filtrations $\mathcal{I}=\{I^n : J^\infty\}_{n \ge 0}$ and $\filJ=\{I^n\}_{n \ge 0}$, the leading coefficient of the eventual quasi-polynomial describing the defect is constant, and that, under a mild height condition on an equigenerated ideal $I$, the coefficient of the second-highest power of $n$ is also constant; see \Cref{theorem on coefficients}.




\begin{theorem}\label[theorem]{growth of sdef vs idef}
Let $R=\mathbb{K}[x_1,\dots,x_r]$ be a standard graded polynomial ring over a field $\mathbb{K}$. Let $\mathcal{I}=\{I_n\}_{n \geq 0}$ and $\filJ=\{J_n\}_{n \geq 0}$ be monomial filtrations in $R$ such that $I_n \subseteq J_n$ for all $n\geq 0$. 
\begin{enumerate}
    \item Set $T_n = \displaystyle \frac{\mathfrak{m}\overline{I}_n \cap \overline{J}_n}{\mathfrak{m}\overline{J}_n}$. If $\displaystyle\lim_{n \rightarrow \infty}\dfrac{\mu(\overline{I}_n)}{\dim_{\mathbb{K}}(T_n)-\mu(\overline{J}_n)}\neq 1$, then  $\displaystyle\lim_{n \rightarrow \infty } \frac{\defect(\mathcal{I},\filJ,n)}{\defect(\overline{\mathcal{I}},\overline{\mathcal{J}},n)} < \infty.$ 
    \item Set $S_n = \displaystyle \frac{I_n \cap \overline{J}_n}{J_n}$. If $\displaystyle\lim_{n \rightarrow \infty}\dfrac{\mu(S_n)}{\defect(\overline{\mathcal{I}},\overline{\mathcal{J}},n)} < \infty$, then $0< \displaystyle\lim_{n \rightarrow \infty}\frac{\defect(\mathcal{I},\filJ,n)}{\defect(\overline{\mathcal{I}},\overline{\mathcal{J}},n)} < \infty $.
    \end{enumerate}
\end{theorem}
\begin{proof}
    In the case that, for $n\gg 0$, $\defect(\mathcal{I},\filJ,n)=0$ or $\defect(\overline{\mathcal{I}},\overline{\mathcal{J}},n)=0$ the result is immediately true. Therefore, we may assume that both $\defect(\mathcal{I},\filJ,n)$ and $\defect(\overline{\mathcal{I}},\overline{\mathcal{J}},n)$ are nonzero for $n\gg 0$.
    
    To prove (1), consider the following short exact sequence :
    $$ 0 \longrightarrow \overline{J}_n \longrightarrow \overline{I}_n \longrightarrow \frac{\overline{I}_n+\overline{J}_n}{\overline{J}_n} \longrightarrow 0.$$
    Tensoring with $\mathbb{K}$ and applying Nakayama's lemma, we get 
    $$\mu(\overline{J}_n)+\ \defect(\overline{\mathcal{I}},\overline{\mathcal{J}},n)=\dim_{\mathbb{K}}T_n+\mu(\overline{I}_n).$$
That is, $\defect(\overline{\mathcal{I}},\overline{\mathcal{J}},n)=\dim_{\mathbb{K}}T_n+\mu(\overline{I}_n)-\mu(\overline{J}_n).$

Let $P(n)=\dim_{\mathbb{K}}T_n-\mu(\overline{J}_n)$ and $F_n= \displaystyle {\defect(\mathcal{I},\filJ,n)}\big/{\defect(\overline{\mathcal{I}},\overline{\mathcal{J}},n)}$ . Clearly, 
$$F_n \leq \dfrac{\mu(I_n)}{\mu(\overline{I}_n)+P(n)}.$$ 
It follows from \cite[Corollary~4.4]{HTH} that the rates of growth of $\mu(I_n)$ and $\mu(\overline{I}_n)$ are both $\ell(\mathcal{I})-1$, where $\ell(\mathcal{I})$ is the analytic spread of $\mathcal{I}$. 
Therefore, if $\deg(P(n))> \ell(\mathcal{I})-1$ then $F_n \rightarrow 0$ as $n \rightarrow \infty$. On the other hand, if $\deg(P(n)) \le \ell(\mathcal{I})-1$, then $F_n$ is bounded above by a finite number. In both cases, we have $\lim_{n \rightarrow \infty} F_n$ is finite.

To prove (2), we consider the following commutative diagram :
\[\begin{tikzcd}
&0\arrow[d,""]&0\arrow[d,""]& \\
0 \arrow[r,""] & J_n \arrow[r, ""] \arrow[d, ""]
& \overline{J_n} \arrow[r, "\phi"] \arrow[d,""] & \displaystyle \frac{\overline{J}_n}{J_n} \arrow[r, ""] \arrow[d, ""] & 0\\
0 \arrow[r,""] & I_n \arrow[r, ""] 
& \overline{I_n} \arrow[r, ""] & \displaystyle \frac{\overline{I}_n}{J_n} \arrow[r, ""] & 0.\\
\end{tikzcd}\]

Using the Snake Lemma, we get the following exact sequence:
$$0 \longrightarrow \dfrac{\overline{J}_n \cap I_n}{J_n}\longrightarrow \dfrac{I_n}{J_n}\longrightarrow \dfrac{\overline{I}_n}{\overline{J}_n}\longrightarrow C_n \longrightarrow 0$$
where $C_n$ denotes the cokernel of the map $\dfrac{I_n}{J_n}\longrightarrow \dfrac{\overline{I}_n}{\overline{J}_n}$. By tensoring with $\mathbb{K}$ and applying Nakayama's lemma, this gives
$$\defect(\mathcal{I},\filJ,n) \leq \defect(\overline{\mathcal{I}},\overline{\mathcal{\filJ}},n)+\mu\left(S_n\right).$$
It follows that $F_n \leq \dfrac{\mu(S_n)}{\defect(\overline{\mathcal{I}},\overline{\mathcal{J}},n)}+1$ and, hence, $$0< \displaystyle\lim_{n \rightarrow \infty}\frac{\defect(\mathcal{I},\filJ,n)}{\defect(\overline{\mathcal{I}},\overline{\mathcal{J}},n)} < \infty \text{ whenever } \displaystyle\lim_{n \rightarrow \infty}\dfrac{\mu(S_n)}{\defect(\overline{\mathcal{I}},\overline{\mathcal{J}},n)} < \infty.  \qedhere$$
\end{proof}

We now prove our second main result, which shows that the leading coefficient and the coefficient of the second highest power of $n$ in the defect function between $\mathcal{I}=I^n: J^\infty$ and $\filJ=\{I^n\}_{n \geq 0}$ is a constant. The proof of this result is inspired by that of \cite[Theorem~1.2]{TJP}. Before we begin the proof, recall that for a multiplicative filtration $\mathcal{F}=\{I_n\}_{n \geq 0}$ of homogeneous ideals with $I_0=R$ and $I \subseteq I_1$, $L^\mathcal{F}=\displaystyle\bigoplus _{n \geq 0}R/I_{n+1}$ has an $R[It]$-module structure (see \cite[2.3]{TJP} for details). This module structure arises as follows: using the fact that $\mathcal{R}(\mathcal{F)}=\displaystyle\bigoplus_{n \geq 0}I_nt^n$ is a graded $R[It]$-module and the short exact sequence
$$ 0 \longrightarrow \mathcal{R}(\mathcal{F})\longrightarrow R[t] \longrightarrow L^\mathcal{F}(-1) \longrightarrow 0,$$ we get an $R[It]$-module structure on $L^\mathcal{F}(-1)$, and hence of $L^\mathcal{F}$.
\begin{theorem}\label[theorem]{theorem on coefficients}
  Let $R=\mathbb{K}[x_1,\dots,x_r]$ be a standard graded polynomial ring. Let $I$ and $J$ be  monomial ideals. Consider two filtrations $\mathcal{I}=\{I^n: J^\infty\}_{n \geq 0}$ and $\filJ=\{I^n\}_{n \geq0}$ of ideals in $R$. Assume that the Rees algebra $\mathcal{R}(\mathcal{I})$ is Noetherian and the defect module 
$$N=\displaystyle \bigoplus_{n \geq 0} \dfrac{I_n+J_n}{J_n}$$  
is a finitely generated module over $\mathcal{R}(\mathcal{I})$. Suppose further that the quasi-polynomial (by \Cref{main theorem 1}) defect function of $(\mathcal{I},\filJ)$ is of the form 
$$\defect(\mathcal{I},\filJ,n) = a_c(n)n^c+a_{c-1}(n)n^{c-1}+ \text{ lower order terms},$$ 
for $n\gg 0$, where $c$ is the degree. Then, the following statements hold:
  \begin{enumerate}
      \item The coefficient $a_c(n)$ is constant.
      \item If $I$ is generated by monomials in the same degree and $\h(I) \geq 2$, then the non-zero values of $a_{c-1}(n)$ are constant.   
  \end{enumerate}     
\end{theorem}

\begin{proof}
(1) The assertion is a direct application of  \cite[Theorem~2.4]{JKV} on the $\mathcal{R}(\mathcal{I})$-module $N=\displaystyle\bigoplus_{n \geq 0}\dfrac{I_n+J_n}{J_n}$, where $I_n = I^n : J^\infty$. 
 
(2) Suppose that the quasi-polynomial of $\defect(\mathcal{I},\filJ,n)$ has degree $c$ and period $\alpha$ (it is worthwhile to note that, by \Cref{HF is quasi polynomial}, the polynomials appearing in the asymptotic quasi-polynomial of $\defect(\mathcal{I},\filJ)$ have the same degree $c$). 
 Let $  \displaystyle L^{\mathcal{I}}=\bigoplus_{n \geq 0} \dfrac{R}{I^{n+1}:J^\infty} \text{ and } L^\filJ= \bigoplus_{n \geq 0} \dfrac{R}{I^{n+1}}$. Then, with $W=\displaystyle\bigoplus_{n \geq 0} \frac{I^{n+1}:J^\infty}{I^{n+1}}$, we have the following s.e.s of $R[It]$-modules
  \begin{equation}\label{SES1}
      0 \longrightarrow W \longrightarrow L^{\filJ} \longrightarrow L^{\mathcal{I}}\rightarrow 0.
   \end{equation}
   
   Using the fact that $I$ is a monomial ideal generated by monomials in same degree, there exists a homogeneous $I$-superficial element $u$ which is also superficial for the filtration $\mathcal{I}$ (see  \cite[\textit{Proof of 1.2.}]{TJP} for details). Let $p=ut$, then by \cite{TJP} the kernel of the map $L^{\mathcal{I}}(-1) \xrightarrow{\cdot p} L^{\mathcal{I}}$ is concentrated in degrees less than or equal to $n_0 \in \mathbb{N}$, such that $(I^{n+1}:u)=I^n$ for all $n \geq n_0$. This gives the following short exact sequence: 
   \begin{equation}\label{SES2}
       0 \longrightarrow \left(\dfrac{W}{pW}\right)_{n \geq n_0+1} \longrightarrow  \left(\dfrac{L^{\filJ}}{pL^{\filJ}}\right)_{n \geq n_0+1} \longrightarrow  \left(\dfrac{L^{\mathcal{I}}}{pL^{\mathcal{I}}}\right)_{n \geq n_0+1} \longrightarrow  0.
   \end{equation}  
   
  Consider the standard graded $\mathbb{K}$ algebra $S=R/(u)$, the $S$-ideals $I'=I/(u)$ and $J' = \dfrac{J+(u)}{(u)}$, and the Noetherian $S$-filtration $\mathcal{I}'=\left\{ \dfrac{I_n+(u)}{(u)}\right\}_{n \geq 0}$. 
  Note that, $\mathrm{grade}(I)=\h(I) \geq 2$ which gives $\mathrm{grade}(I') \geq 1$. Let $\filJ' = \left\{\dfrac{I^n+(u)}{(u)}\right\}_{n \ge 0}$. Using \cite{JKV}, we have  $\dim_{\mathbb{K}}\left({(I^n:J^\infty)}/{I^n}\right)$ and $\dim_{\mathbb{K}}\left(\dfrac{(I^n:J^\infty)+(u)}{I^n+(u)}\right)$ are constants for $n \gg 0$, we denote these constants by $r$ and $s$ respectively. Furthermore, by \Cref{quasi-poly}, $\defect(\mathcal{I}',\filJ',n)$ is eventually a quasi-polynomial of period $\alpha'$ and degree $l$, i.e.,
  $$\defect(\mathcal{I}',\filJ',n) = b_l(n)n^l+ \text{ lower order terms. }$$ By (1), $b_l(n)$ is a constant $b$ for all values of $n$. 
  
  For $n \geq n_0+1$, consider the multiplication map:
  $$f: \dfrac{I^{n-1}:J^\infty}{I^{n-1}} \longrightarrow \dfrac{I^n: J^\infty }{I^n},$$ defined by $f(x+I^{n-1})=ux+I^n$. We first compute the kernel of the map $f$. An element $x+I^{n-1} \in \mathrm{Ker}(f)$ if and only if $ux \in I^n$ and hence, we have:
  $$\mathrm{Ker}(f)=\dfrac{(I^{n-1}: J^\infty) \cap (I^n:u) }{I^{n-1}}.$$ Since $u$ is an $I$-superficial element and $n \geq n_0+1$, we have $I^n:u=I^{n-1}$. Therefore, $\mathrm{Ker}(f)=0$. Note that the image of the map $f$ is given by:
  $$ \mathrm{Im}(f)=\dfrac{u(I^{n-1}:J^{\infty})+I^n}{I^n}.$$
  Consequently, $\mathrm{Coker}(f)=\dfrac{I^n:J^\infty}{u(I^{n-1}:J^\infty)+I^n} \cong \dfrac{(I^n: J^\infty)+(u)}{I^n+(u)}$ (from (\ref{SES1}) and (\ref{SES2})). Combining the above, we get the following short exact sequence for all $n \gg 0$: 
$$0 \longrightarrow \frac{I^{n-1}:J^\infty}{I^{n-1}}\longrightarrow \frac{I^n : J^\infty}{I^n} \longrightarrow \frac{(I')^n : (J')^\infty}{(I')^n} \longrightarrow 0.$$ 
Therefore, $r \geq s$. If $r>s$, then $c=0$ and there is nothing to prove. 

Now, suppose that $r=s$. We have $l=c-1$. Finally, since $\defect(\mathcal{I},\filJ,n)-\defect(\mathcal{I},\filJ,n-1)\leq \defect(\mathcal{I}',\filJ',n)$ we get the following expression:
$$ac+a_{c-1}(n)-a_{c-1}(n-1) \leq b \text{ for all }n\gg 0,$$
where $a$ is the constant highest degree coefficient in the quasi-polynomial $\defect(\mathcal{I},\filJ,n)$ for $n\gg 0$. Hence, for $n\gg 0$, we get $a_{c-1}(n)\leq \gamma n + \lambda$, for some constants $\gamma$ and $\lambda$ (see \cite[Theorem~4.1.2.]{bruns1998cohen}). Suppose that $a_{c-1}(n)$ is non-zero, then $a_{c-1}(n)$ is a periodic function of $n$. It follows that $\gamma=0$, and so $a_{c-1}(n)$ is a constant for sufficiently large values of $n$.
\end{proof}

Example \ref{Example1} demonstrates that the expression for $a_{c-1}(n)$ in \Cref{theorem on coefficients} may be zero for certain values of $n \gg 0$. However, the following example illustrates \Cref{theorem on coefficients} in exhibiting that non-zero values of $a_{c-1}(n)$ are constant. 

\begin{eg}[{\cite[Example~4.11]{MM}}] \label[eg]{Example2}
  Let $I = (x_1x_2, x_2x_3,x_3x_4,x_4x_5,x_5x_1) \subseteq \mathbb{K}[x_1, \dots, x_5]$. Then, for the filtrations $\mathcal{I}=\{I^{(n)}\}_{n \geq 0}$ and $\filJ=\{I^n\}_{n \geq 0}$, their defect function is  
  \[
\defect(\mathcal{I},\mathcal{J}, n) =
\begin{cases}
\frac{15}{2}k^3 - \frac{15}{2}k^2 + 1 & \text{if } n = 3k, \\[1em]
\frac{15}{2}k^3 - \frac{5}{2}k & \text{if } n= 3k + 1, \\[1em]
\frac{15}{2}k^3 + \frac{15}{2}k^2 & \text{if } n = 3k + 2.
\end{cases}
\]
\end{eg}

We end the paper with an example which shows that $I$ is equigenerated in \Cref{theorem on coefficients} is necessary for the non-zero values of $a_{c-1}(n)$ to be constant. 

\begin{eg}[{\cite[Example~6.12]{BCMM}}] \label[eg]{Example3}
   Let $I=(xyz,x^2z,y^3x,z^4y) \subseteq \mathbb{K}[x,y,z]$. Then for the filtrations $\mathcal{I}=\{I^{(n)}\}_{n \geq 0}$ and $\filJ=\{I^n\}_{n \geq 0}$, their defect function is
   $$\defect(\mathcal{I},\filJ,n)=\begin{cases}
       0, \ \ \ n=1\\[0.5em]
       \dfrac{11}{5}n-2, \ \ \ n\equiv 0 \ (\text{mod } 5)\\[0.5em]
        \dfrac{11}{5}n-\dfrac{6}{5},\ \ \ n\equiv 1 \ (\text{mod } 5)\\[0.5em]
         \dfrac{11}{5}n-\dfrac{7}{5},\ \ \ n\equiv 2 \ (\text{mod } 5)\\[0.5em]
          \dfrac{11}{5}n-\dfrac{3}{5},\ \ \ n\equiv 3 \ (\text{mod } 5)\\[0.5em]
           \dfrac{11}{5}n-\dfrac{4}{5},\ \ \ n\equiv 4 \ (\text{mod } 5)
   \end{cases}$$
\end{eg}


\vspace*{1cm}

\subsection*{Acknowledgements.}  The first author is supported by the Indian Institute of Technology Kharagpur startup grant. The
second author is partially supported by a Simons Foundation grant.
 The third  author is thankful to the Government of India for supporting him in this work through the Prime Minister's Research Fellowship. A part of this project was done while the third author visited the second author at Tulane University. The authors thank Tulane University for its hospitality. The authors acknowledge the use of the computer algebra system Macaulay2 \cite{M2} for testing their computations.
\bibliographystyle{abbrv}
\bibliography{refs}

\end{document}